\documentclass{amsart}

\RequirePackage{color}

\def\smallskip{\vskip\smallskipamount}
\def\medskip{\vskip\medskipamount}
\def\bigskip{\vskip\bigskipamount}

\usepackage{amsmath,amsthm,bm,mathrsfs,amscd,tikz}
\usepackage{ytableau}
\usepackage{comment}
\usepackage{mathdots}
\usepackage[all]{xy}
\usepackage{graphicx}
\usepackage[colorinlistoftodos]{todonotes}
\usepackage{enumerate}
\usepackage{feynmp-auto}
\usepackage[english]{babel}
\usepackage[utf8]{inputenc}
\usepackage{float}
\usepackage{tikz}
\usepackage{tikz-cd}
\usepackage{amssymb}
\usepackage{mathtools}
\usepackage{soul}
\usetikzlibrary{decorations.pathmorphing}

\usepackage[utf8]{inputenc}

\newtheoremstyle{thmstyle}{}{}{\itshape}{}{\bfseries}{ }{5pt}{}
\newtheoremstyle{exstyle}{}{}{}{}{\bfseries}{ }{5pt}{}
\newtheoremstyle{defstyle}{}{}{}{}{\bfseries}{ }{5pt}{}
\newtheoremstyle{remstyle}{}{}{}{}{\bfseries}{ }{5pt}{}

\theoremstyle{thmstyle}
\newtheorem{thm}{Theorem}[section]
\newtheorem{theorem}[thm]{Theorem}

\newtheorem{lemma}[thm]{Lemma}

\newtheorem{proposition}[thm]{Proposition}

\newtheorem{corollary}[thm]{Corollary}

\newtheorem{conjecture}{Conjecture}[section]

\theoremstyle{exstyle}

\theoremstyle{defstyle}

\newtheorem{def-prop}[thm]{Definition-Proposition}
\newtheorem{def-lem}[thm]{Definition-Lemma}
\newtheorem{rem-convention}[thm]{Remark-Convention}
\newtheorem{def-note}[thm]{Definition-Notation}

\theoremstyle{remstyle}

\newtheorem{remark}[thm]{Remark}

\theoremstyle{remstyle}

\newcommand{\Hom}{\operatorname{Hom}}

\newcommand{\taurigid}{\operatorname{\tau-rigid}}

\DeclareMathOperator*{\End}{End}
\DeclareMathOperator*{\rad}{rad}

\DeclareMathOperator*{\Image}{Im}		
\DeclareMathOperator*{\modu}{mod}

\DeclareMathOperator*{\proj}{proj}

\DeclareMathOperator*{\ann}{ann}
\newcommand{\Z}{\mathcal{Z}}

\DeclareMathOperator*{\ind}{ind}

\DeclareMathOperator*{\brick}{brick}

\DeclareMathOperator*{\pd}{pd}
\DeclareMathOperator*{\id}{id}

\DeclareMathOperator*{\GL}{GL}

\DeclareMathOperator*{\rep}{rep}

\DeclareMathOperator*{\Irr}{Irr}

\DeclareMathOperator*{\supp}{supp}

\newcommand{\cupdot}{\mathbin{\mathaccent\cdot\cup}}

\makeatletter
\newcommand{\doublewidetilde}[1]{{%
  \mathpalette\double@widetilde{#1}%
}}
\newcommand{\double@widetilde}[2]{%
  \sbox\z@{$\m@th#1\widetilde{#2}$}%
  \ht\z@=.9\ht\z@
  \widetilde{\box\z@}%
}
\makeatother

\begin{document}


\title{Hom-orthogonal modules and brick-Brauer-Thrall conjectures}

\author[Kaveh Mousavand, Charles Paquette]{Kaveh Mousavand, Charles Paquette} 
\address{Unit of Representation Theory and Algebraic Combinatorics, Okinawa Institute of Science and Technology, Okinawa, Japan}
\email{mousavand.kaveh@gmail.com}
\address{Department of Mathematics and Computer Science, Royal Military College of Canada, Kingston ON, Canada}
\email{charles.paquette.math@gmail.com}

\begin{abstract}
For finite dimensional algebras over algebraically closed fields, we study the sets of pairwise Hom-orthogonal modules and obtain new results on some open conjectures on the behaviour of bricks and several related problems, which we generally refer to as brick-Brauer-Thrall (bBT) conjectures. 
Using some algebraic and geometric tools, and in terms of the notion of Hom-orthogonality, we find necessary and sufficient conditions for the existence of infinite families of bricks of the same dimension. 
This sheds new light on the bBT conjectures and we prove some of them for new families of algebras. 
Our results imply some interesting algebraic and geometric characterizations of brick-finite algebras as conceptual generalizations of local algebras. We also verify the bBT conjectures for any algebra whose Auslander-Reiten quiver has a generalized standard component, which particularly extends some results of Chindris-Kinser-Weyman on the algebras with preprojective components.
\end{abstract}

\maketitle

\tableofcontents

\section{Introduction and Main Results}\label{Section: Introduction}
In the following, $A$ always denotes a finite dimensional unital associative algebra over an algebraically closed field $k$. Without loss of generality, we further assume $A$ is basic and connected of rank $n$. A module always means a finitely generated left $A$-modules, considered up to isomorphism.
A (finite dimensional) module $X$ is said to be a \emph{brick} (also known as Schur representation) if $\End_A(X)\simeq k$, and $A$ is called \emph{brick-finite} if it admits only a finite number of (isomorphism classes of) bricks.
There has been a long history of study of bricks in various domains of representation theory and related topics. Many of the more recent studies of bricks are motivated by their connections to wide subcategories and torsion pairs, $\tau$-tilting theory, tilting and silting theory, stability conditions and wall-chamber structures, geometry of representation varieties, realizations of $g$-vector fans, just to mention a few.

\medskip

In this work, we pursue two goals. First, we adopt a less technical viewpoint and study bricks through the lens of Hom-orthogonal modules. Second, we develop our earlier studies on some open conjectures on the behaviour of brick-infinite algebras. 
Recall that a pair of non-zero $A$-modules $X$ and $Y$ is said to be \emph{Hom-orthogonal} if $\Hom_A(X,Y)=\Hom_A(Y,X)=0$.  Henceforth, by a Hom-orthogonal set $\mathcal{S}$ we mean a non-empty set of non-zero modules in $\modu A$, such that elements of $\mathcal{S}$ are pairwise non-isomorphic and Hom-orthogonal, that is, $\Hom_A(X,Y)=\Hom_A(Y,X)=0$, for any pair of distinct $X$ and $Y$ in $\mathcal{S}$.
As shown below, the study of Hom-orthogonality from algebraic and geometric viewpoints provides new insights into the behaviour of bricks and some challenging open conjectures.
To put our work in perspective, we summarize our results in three main theorems, each of which consists of several assertions shown in the following sections. The notions and terminology that are not explicitly defined in this section will be recalled throughout the note.

\medskip

Our first theorem establishes a useful relationship between Hom-orthogonal sets of arbitrary modules and Hom-orthogonal sets of bricks. Let us recall that, as in \cite{As}, a Hom-orthogonal set of bricks in $\modu A$ is often called a \emph{semibrick} of $A$, and the size of a semibrick is defined to be the number of bricks belonging to it.

\medskip

\textbf{Theorem A:}
For algebra $A$, if $\mathcal{S}=\{X_i\}_{i\in I}$ is a Hom-orthogonal set in $\modu A$, then there exists a Hom-orthogonal set of bricks $\mathcal{S}_b=\{B_i\}_{i\in I}$, where each $B_i$ is a submodule and a quotient module of $X_i$. 
Moreover, we have 
\begin{enumerate}
    \item If $A$ is a brick-finite algebra of rank $n$, then there is no Hom-orthogonal set of modules with $n+1$ elements. In addition, there exists a unique semibrick of size $n$, and it is given by the simple modules in $\modu A$.
    \item If $A$ admits an infinite set of Hom-orthogonal modules of dimension $d'$, then there exists an infinite semibrick consisting of bricks of dimension $d \leq d'$.
\end{enumerate}

\medskip
Each non-empty subset of a Hom-orthogonal set is evidently Hom-orthogonal. In particular, any algebra with an infinite semibrick admits infinitely many semibricks of size $n$.
The above theorem poses an \emph{a priori} simpler version of the Semibrick Conjecture (Conjecture \ref{Conj:Semibrick Conjecture}) as follows: $A$ is brick-finite if and only if it admits a unique semibrick of size $n$. For more details, see Sections \ref{Section: Preliminaries and Background} and \ref{Section: Hom-orthogonal modules and brick-finiteness}.

\medskip

The next proposition summarizes some of our  results on Hom-orthogonality in a more geometric setting. More precisely, for each representation variety $\rep(A,\underline{d})$, we study the interaction between the orbits of an irreducible component $\Z$ of $\rep(A,\underline{d})$. The undefined notations and terminology are recalled in Sections \ref{Section: Preliminaries and Background} and \ref{Section: Orthogonal orbits and irreducible components}.

\medskip

\textbf{Proposition B:}
Let $A$ be an algebra and $\Z$ an irreducible component of $A$. If $\Z$ has a dense orbit, then for each $X$ and $Y$ in $\Z$, we have $\Hom_A(X,Y)\neq 0$.
Furthermore, the following are equivalent:
\begin{enumerate}
    \item An irreducible component of $A$ contains a pair of orthogonal orbits;
    \item $A$ admits infinitely many bricks of the same dimension.
\end{enumerate}
In fact, if $A$ is brick-finite, for each irreducible component $\Z$, and every pair $X$ and $Y$ in $\Z$, we have $\Hom_A(X,Y)\neq 0$.
\medskip

The first assertion is proved in Lemma \ref{No orhtogonal modules when a dense orbit}, and the complete proof of Proposition B is given in Section \ref{Section: Orthogonal orbits and irreducible components}, where we also make some observations on connections to some recent related work in the geometric setting. Meanwhile, let us note that $A$ is a local algebra if and only if, for any pair of non-zero modules $X$ and $Y$, we have $\Hom_A(X,Y)\neq 0$. Hence, the last assertion of Proposition B implies that brick-finite algebras can be seen as a conceptual generalization of local algebras (see Corollary \ref{Cor: geometric version of local algebras}). This further allows us to rephrase the Second brick-Brauer-Thrall conjecture (Conjecture \ref{Conj: 2ndbBT}) as follows: $A$ is brick-finite if and only if for each brick component $\Z$, and every pair $X$ and $Y$ in $\Z$, we have $\Hom_A(X,Y)\neq 0$.

\medskip

Our second main result is concerned with some interesting connections between the generalized standard components in the Auslander-Reiten quivers and the infinite families of bricks of the same dimension.
In particular, for those algebras which admit a generalized standard component, the following theorem verifies several open brick-Brauer-Thrall conjectures listed in Section \ref{Subsection:bBT Conjs}.
The necessary materials and terminology are recalled in Sections \ref{Section: Preliminaries and Background} and \ref{Section: Bricks and generalized standard components}.

\medskip

\textbf{Theorem C:}
Let $A$ be an algebra whose Auslander-Reiten quiver contains a generalized standard component. Then, either $A$ is representation-finite, or else all of the following properties hold for $A$:
\begin{enumerate}
    \item $A$ admits an infinite family of bricks of the same dimension;
    \item $A$ admits an infinite semibrick;
    \item There is a rational ray outside of the $\tau$-tilting fan of $A$;
    \item For some $\theta \in K_0(\proj A)$, there is an infinite family of $\theta$-stable bricks of the same dimension.
\end{enumerate}

\medskip

We remark that those algebras that admit a preprojective component have been crucial in the study of representation theory of algebras. Meanwhile, it is known that each preprojective component is generalized standard. The previous theorem particularly extends the setting and main results of \cite{CKW}, where similar studies are conducted over those algebras that have preprojective components.

\medskip

The stable tubes in the Auslander-Reiten quiver $\Gamma_A$ are known to be decisive in the representation theory of $A$. For instance, the Second Brauer-Thrall conjecture (now theorem), together with some fundamental results of Crawley-Boevey \cite{C-B}, imply that a tame algebra $A$ admits a $1$-parameter family of indecomposables if and only if $\Gamma_A$ has a stable tube. This is the case, if and only if $\Gamma_A$ contains an infinite family of homogeneous tubes whose mouths consist of modules of the same dimension.
As a consequence of the previous theorem, and using some classical results on tame algebras, we obtain a brick-analogue of these important results over tame algebras.
\medskip

\textbf{Corollary D:}
For any tame algebra $A$, the following are equivalent:
\begin{enumerate}
    \item $A$ admits a $1$-parameter family of bricks;
    \item $\Gamma_A$ contains a (generalized) standard stable tube;
    \item $\Gamma_A$ contains an infinite family of pairwise orthogonal homogeneous tubes whose mouth are of the same dimension;
    \item For some $\theta \in K_0(\proj A)$, there is a rational curve of $\theta$-stable modules.
\end{enumerate}

\medskip
As explained in the following sections, the Second brick-Brauer-Thrall ($2nd$ bBT) conjecture has played a central role in our recent studies.
Although this problem can be treated in more technical settings (for instance, see \cite{Mo1, Mo2}, \cite{MP1, MP2, MP3}, \cite{STV}, \cite{Pf1, Pf2}), we believe the novelty of the $2nd$ bBT conjecture is more apparent from its non-technical articulation: \emph{If $A$ is brick-infinite, then there exists an infinite family of bricks of the same dimension}. 
In fact, some of the other open conjectures listed in Section \ref{Subsection:bBT Conjs} either follow from the $2nd$ bBT (see Remark \ref{Rem:on semibrick conjecture}), or are shown to be equivalent to it in certain settings (see \cite{MP2} and \cite{Pf2}).  

\medskip

Let us remark that the basic notion of Hom-orthogonality that we adopt in this work leads to some equivalent conditions that should provide fresh impetus to the study of several challenging problems.
In particular, for an algebra $A$, our results imply that the following are equivalent:
\begin{enumerate}
    \item There exists some $d\in \mathbb{Z}_{>0}$ for which $\brick(A,d)$ is an infinite set;
    \item $A$ admits an infinite semibrick consisting of bricks of dimension $d$;
    \item A component $\Z \in \Irr(A,d')$ contains $X$ and $Y$ with $\Hom_A(X,Y)=0$;
    \item A component $\Z\in \Irr(A,d')$ has infinitely many orthogonal orbits. 
\end{enumerate}

We observe that, after applying some geometric arguments, some of the above implications can also be concluded from \cite[Theorem 1.5]{G+}. We further note that, in the above set of equivalences, if one (and thus all) of these conditions holds, we can choose $d=d'$ so that all of the statements hold for $d$.
Because these implications are consequences of our main results, separate proofs are not provided. Instead, let us remark that $(1)\rightarrow (2)$ follows from Proposition \ref{Prop:2nd bBT gives an infinite family of orthogonal bricks}, and $(2) \rightarrow (3)$ follows from a standard geometric observation, and $(3)\rightarrow (4) \rightarrow (1)$ follow from Theorem \ref{Thm:1-side vanishing Hom of two modules}.

\medskip

We finish this section with an outline of the content of the following sections.
In Section \ref{Section: Preliminaries and Background} we recall some background materials and the open conjectures related to the scope of our studies, which we generally call the brick-Brauer-Thrall (bBT) conjectures.
In Section \ref{Section: Hom-orthogonal modules and brick-finiteness}, we consider the notion of Hom-orthogonality in the category of finite dimensional modules and prove Theorem A.
In Section \ref{Section: Orthogonal orbits and irreducible components}, we treat the notion of orthogonality in the irreducible components of representation varieties and prove Proposition B.
In Section \ref{Section: Bricks and generalized standard components}, we consider the connection between the components of the Auslander-Reiten quiver and the behavior of bricks, and prove Theorem C and Corollary D.

\section{Preliminaries and Background}\label{Section: Preliminaries and Background}

This section mainly contains some tools and known results that we freely use througout the paper. All the the standard or rudimentary materials can be found in \cite{ARS}, \cite{ASS} and \cite{SS}. For the more technical results, We provide references. 

\medskip

\noindent \textbf{Notations and setting}:
Throughout, $k$ is assumed to be an algebraically closed field and $A$ always denotes a finite dimensional, basic, connected, unital associative $k$-algebra of rank $n$.
By $\modu A$ we denote the category of finitely generated left $A$-modules.
Due to our assumptionss on $A$, there exists a unique quiver $Q$ with $n$ vertices, and an admissible ideal $I$ in $kQ$, such that $A \simeq kQ/I$.
Hence, each $M$ in $\modu A$ can be seen as a representation of the bound quiver $(Q,I)$, whose dimension vector in $ \mathbb{Z}^n_{\geq 0}$ is denoted by $\underline{\dim} M$.
By $\proj A$, we denote the full exact subcategory of $\modu A$ whose objects are projective modules, and $K^b({\rm proj}A)$ denotes the homotopy category of bounded complexes of $\proj A$. Moreover, $D^b(\modu A)$ denotes the bounded derived category of $\modu A$. By suitable choices of bases, $K_0(K^b({\rm proj}A))$ and $K_0(\rm proj A)$, as well as $K_0(D^b(\modu A))$, are identified with $\mathbb {Z}^n$.

For each $d\in \mathbb{Z}_{\geq 0}$, let $\ind(A,d)$ denote the set of all indecomposable modules of dimension $d$, and define $\ind(A)$ to be the union of all $\ind(A,d)$, for $d\in \mathbb{Z}_{\geq 0}$. Then, $A$ is called \emph{representation-finite} if $\ind(A)$ is a finite set.
Similarly, by $\brick(A,d)$, we denote the set of all of bricks of dimension $d$, and
define $\brick(A)$ to be the union of all $\brick(A,d)$, for all $d\in \mathbb{Z}_{\geq 0}$. Recall that, $A$ is said to be \emph{brick-finite} provided $\brick(A)$ is a finite set. Evidently, $\brick(A,d)\subseteq \ind(A,d)$, for all $d\in \mathbb{Z}_{\geq 0}$.
We also remark that, for a given family, say $F$, we say a property holds for ``almost all" elements of $F$ to mean it holds for ``all but a finite number of elements" in $F$.

\medskip

\subsection{Representation varieties and their components}\label{Subsection: Rep. varieties and components}
For algebra $A$ and a dimension vector $\underline{d} \in (\mathbb{Z}_{\geq 0})^n$, by $\rep(A,\underline{d})$ we denote the affine variety which parametrizes the modules $X$ in $\modu A$ whose dimension vector is $\underline{d}$. Recall that the general linear group $\GL(\underline{d})$ acts on $\rep(A,\underline{d})$ via conjugation and turns it into an affine variety. In particular, for $X$ in $\rep(A,\underline{d})$, let $O_X$ be the $\GL(\underline{d})$-orbit of $X$ and by $\overline{O_X}$ we denote the closure of $O_X$. 

\medskip

Let $\Irr(A,\underline{d})$ denote the set of all irreducible components of $\rep(A,\underline{d})$, and define $\Irr(A):= \bigcup_{\underline{d} \in \mathbb{Z}_{\geq 0}^n} \Irr(A,\underline{d})$. 
For each $\mathcal{Z} \in \Irr(A)$, define $$c(\Z):=\min \{\dim(\Z) - \dim(O_Z) \,| Z \in \Z \}.$$ 
In other words, $c(\Z)$ denotes the generic number of parameters of $\Z$. In particular, $c(\Z)=0$ if and only if $\Z$ contains an open orbit (i.e, $\Z=\overline{O_M}$, for some $M \in \Z$).

\medskip

By $T_M(\underline{d})$ and $T_M(O_M)$ we respectively denote the tangent spaces of $M$ in $\rep(A,\underline{d})$ and in $O_M$. Observe that $T_M(O_M)$ is a subspace of $T_M(\underline{d})$ as a vector space. The following lemma is well-known and plays an important role in the next sections. For a proof of this fact and more details on the geometric setting, we refer to \cite{De} and references therein.

\begin{lemma}\label{Lem: tanget spaces-isomorphism}
With the above notations, there always exists a monomorphism
$$T_M(\underline{d})/T_M(O_M) \longrightarrow {\rm Ext}^1_A(M,M).$$
Moreover, we get the isomorphism $T_M(\underline{d})/T_M(O_M) \simeq {\rm Ext}^1_A(M,M)$, provided that ${\rm Ext}^2_A(M,M)=0$.
\end{lemma}

The well-known Voigt's lemma is a consequence of Lemma \ref{Lem: tanget spaces-isomorphism} and asserts that if ${\rm Ext}^1_A(M,M)=0$, then $O_M$ is open. Observe that the converse of Voigt's lemma does not hold in general, unless we specify that $O_M$ is scheme-theoretically open. Namely, we can have that $O_M$ is open in the variety $\rep(A,\underline{d})$ while $M$ is non-rigid (for instance, consider the simple module over the local algebra $A=k[x]/\langle x^2 \rangle$).

\medskip

\medskip

For each $\Z$ in $\Irr(A)$, it is well-known that the set of bricks in $\Z$ form an open subset, and $\Z$ is called a \emph{brick component} if this subset is non-empty.
From a standard geometric argument, it follows that $\Z$ either contains a single orbit of bricks, or else infinitely many distinct orbits of bricks belong to $\Z$. This respectively implies $\Z=\overline{O_M}$, or else $\Z=\overline{\bigcup_{M}O_M}$, where $M$ runs through the set of all bricks in $\Z$.
Note that the latter case holds if and only if $\brick(A,d)$ is an infinite set, for some $d\in \mathbb{Z}_{>0}$. That being the case, $A$ is evidently brick-infinite. However, the converse is not known and is the content of one of the open conjectures that we treat in this work (see Conjecture \ref{Conj: 2ndbBT}).

\medskip

The following proposition on brick components will be handy in the next sections. Although the assertion is immediate from the above facts, we present a short proof.

\begin{proposition}\label{Prop: bricks with pd=1}
Let $A$ be an algebra and $M$ be a brick of dimension $d$ in $\modu A$. If $\pd_A(M)\leq 1$, then either $M$ is $\tau$-rigid, or else $\brick(A,d)$ is an infinite set.
\end{proposition}
\begin{proof}
First we observe that, for an arbitrary irreducible component $\Z \in \Irr(A,\underline{d})$ and some $X$ in $\Z$, we have $O_X$ is an open orbit in $\Z$ if and only if $T_X(\underline{d}) = T_X(O_X)$, which is the case if and only if $\Z=\overline{O_X}$.

Let $\Z \in \Irr(A,\underline{d})$ be an irreducible component that contains the brick $M$ with $\pd_A(M)\leq 1$. From Lemma \ref{Lem: tanget spaces-isomorphism}, together with the Auslander-Reiten duality, it is immediate that $T_M(\underline{d})/T_M(O_M)\simeq \Hom_A(M,\tau M)$, as an isomorphism of vector spaces. 
Now the conclusion is clear. In particular, $M$ is $\tau$-rigid if and only if $O_M$ is open. Thus, if the orbit of $M$ is not open, then $\brick(A,d)$ is an infinite set.
\end{proof}

Recall that $K_0(D^b(\modu A))\simeq \mathbb {Z}^n$, and by the choice of the standard basis $\{[S_1], \ldots, [S_n]\}$ given by the simple modules in $\modu A$, for $M$ in $\modu A$, one can identify $\underline{\dim}M$ with the corresponding element $[M]$ in $K_0(D^b(\modu A))$. 
Observe that there is a natural identification of the Grothendieck groups $K_0(K^b({\rm proj}A))$ and $K_0(\rm proj A)$. Consider the basis $\{[P_1], \ldots, [P_n]\}$, given by the indecomposable projective modules. Then, we have the isomorphism $K_0(\rm projA) \cong \mathbb{Z}^n$.
Let us use $K_0(\rm proj A)_{\mathbb{Q}}$ and $K_0(\rm proj A)_{\mathbb{R}}$ to respectively denote $K_0(\rm proj A) \otimes_{\mathbb{Z}} \mathbb{Q}$ and $K_0(\rm proj A) \otimes_{\mathbb{Z}} \mathbb{R}$. Obviously, $K_0(\rm proj A)_{\mathbb{Q}}\cong \mathbb{Q}^n$ and $K_0(\rm proj A)_{\mathbb{R}}\cong \mathbb{R}^n$.
For a given $\theta$ in $K_0(\proj A)_{\mathbb{R}}$, and each $M \in \modu A$, define $\theta([M]):=\theta \cdot \underline{\dim} M$. In particular, for $\theta \in K_0(\proj A)$, a module $X \in \modu A$ is called \emph{$\theta$-stable} if $\theta([X])=0$ and, furthermore, for any nonzero proper submodule $Y$ of $X$, we have $\theta([Y])<0$. 

\medskip

It is well-known that if $X$ is $\theta$-stable, for some $\theta \in K_0(\proj A)$, then $X$ is a brick. However, the converse is not necessarily true for arbitrary bricks. Meanwhile, if $X$ is a brick which is homogeneous (i.e, $\tau X=X$), one can find $\theta \in K_0(\proj A)$ such that $X$ is $\theta$-stable. 
To explicitly describe this stability parameter $\theta$, following \cite[Section 3]{Do}, to each $M \in \modu A$, we associate $\theta_M \in K_0(\proj A)$, given by
$$\theta_M(-):=\dim_k\Hom_A(M,-)-\dim_k\Hom_k(-,\tau_A M).$$
Equivalently, for each $N \in \modu A$, we have $$\theta_M([N]):=\dim_k\Hom_A(P_0,N)-\dim_k\Hom_k(P_1,N),$$ where $P_1\rightarrow P_0 \rightarrow M\rightarrow 0$ is a minimal projective presentation of $M$. Thus, $\theta_M([N])$ is well-defined.
As already noted in \cite{CKW}, if $X$ is a homogeneous brick, one can easily show that $X$ is in fact $\theta_X$-stable. We mention this in the following lemma and use it in the next sections. For a proof, we refer to \cite[Lemma 2.5]{CKW}.
\begin{lemma}\label{Lem: Homogeneous brick is stable}
Let $X$ be a homogeneous brick in $\modu A$. Then, $X$ is $\theta_X$-stable. 
\end{lemma}

\subsection{brick-Brauer-Thrall conjectures}\label{Subsection:bBT Conjs}

As briefly explained below, the algebraic and geometric behaviours of bricks play a central role in our recent studies. Some of these problems can be seen as the modern analogues of the celebrated Brauer-Thrall (BT) conjectures and, consequently, we refer to them as the \emph{brick-Brauer-Thrall} (bBT) conjectures. Here we collect some open conjectures that closely relate to our recent studies and current work. For details, historical remarks, and more refined versions of the classical BT conjectures (now theorems), we refer to \cite{Bo}.

\begin{conjecture}[Second brick-Brauer-Thrall Conjecture -- 2nd bBT]\label{Conj: 2ndbBT}
An algebra $A$ is brick-finite if and only if $\brick(A,d)$ is a finite set, for all $d \in \mathbb{Z}_{>0}$.
\end{conjecture}

The above conjecture first appeared in 2019, in the preprint of \cite{Mo2}, and it was settled for some families of algebras (also see \cite[Chapter 6]{Mo1}). In his PhD dissertation, the first-named author originally treated the $2nd$ bBT as a conceptual linkage between the study of $\tau$-tilting finite algebras in \cite{DIJ} and some algebro-geometric properties of Schur representations in \cite{CKW}. This conjecture also later appeared in \cite{STV} and was settled for special biserial algebras. In our previous work, we introduced a systematic treatment of the $2nd$ bBT via $\tau$-tilting theory and showed a reduction theorem \cite{MP3}, proved a stronger version of this conjecture over all biserial algebras \cite{MP1}, and further studied the $2nd$ bBT conjecture for some other families of tame algebras \cite{MP2}. In \cite{Pf1}, the author has recently verified the $2nd$ bBT conjecture for some new families of algebras.

\medskip

Another conjecture that closely relates to the scope of this work is in terms of Hom-orthogonal bricks. Following \cite{As}, a set of Hom-orthogonal bricks in $\modu A$ is called a \emph{semibrick} of $A$. To our knowledge, the following conjecture was first stated in \cite[Conjecture 5.12]{En1}, and it was originally motivated by the study of some extension-closed subcategories. As discussed in the following sections, this conjecture follows from the $2nd$ bBT conjecture (see Remark \ref{Rem:on semibrick conjecture}). 

\begin{conjecture}[Semibrick Conjecture]\label{Conj:Semibrick Conjecture}
Every brick-infinite algebra admits an infintie semibrick.
\end{conjecture}

The next conjecture is originally inspired by the geometry of the fan induced by $\tau$-tilting theory. As described in \cite{DIJ}, to an algebra $A$, one can associate a polyhedral fan constructed by the $g$-vectors of the indecomposable $\tau$-rigid modules and the shifts of the indecomposable projective modules, which we simply call the \emph{$\tau$-tilting fan} of $A$, and denote it by $\mathcal{F}_A$. 
This is a conceptual analogue of the $g$-vector fan studied in cluster algebras. If $A$ is of rank $n$, then it is shown that $A$ is brick-finite if and only if $\mathcal{F}_A$ is a complete fan in $\mathbb{R}^n$, namely, $\mathcal{F}_A=\mathbb{R}^n$. 
Recently there has been an extensive study of those algebras which are not brick-finite but their $\tau$-tilting fan manifest interesting properties (see \cite{PY}, \cite{AI}, and the references therein).
The following statement, thanks to the characterization of $\tau$-reduced components by Plamondon in \cite{Pl}, is equivalent to a question that was first posed in the work of Demonet \cite{De} and it is open in full generality. There has been some recent work on the connections between the $2nd$ bBT conjecture and the following conjecture (see \cite{MP2} and \cite{Pf2}). 

\begin{conjecture}[Demonet's Conjecture]\label{Conj: Demonet's Conjecture}
Let $A$ be an algebra of rank $n$. If $A$ is brick-infinite, there is a ray in $\mathbb{Q}^n$ that does not belong to the $\tau$-tilting fan of $A$.
\end{conjecture}

The next conjecture that we mention here predates the above conjectures, but closely relates to our studies of bricks. It was originally phrased in a slightly more technical language, and in terms of dimension of the algebras of semi-invariants on irreducible components in $\Irr(A)$ (see \cite[Conjecture 5.4]{CKW}). 
However, for simplicity, below we present it in the language of stable modules. Recall that, for $\theta \in K_0(\proj A)\simeq \mathbb{Z}^n$, a module $M$ in $\modu A$ is \emph{$\theta$-stable} if $\theta(\underline{\dim}M)=0$, and moreover $\theta(\underline{\dim}N)<0$, for each nonzero submodule $N \subsetneq M$.

\begin{conjecture}[Stable brick Conjecture]\label{Conj:CKW}
If $A$ admits an infinite family of bricks of the same dimension, then for some $\theta \in  K_0(\proj A)$, there are infinitely many $\theta$-stable modules of the same dimension.
\end{conjecture}

In \cite{CKW}, where (an equivalence of) the above conjecture originally appears, the authors prove it for all tame algebras (\cite[Theorem 5.1]{CKW}). Moreover, this conjecture is known to hold for any algebra that admits a preprojective component in its Auslander-Reiten quiver (see Section \ref{Section: Bricks and generalized standard components} for more details).

\medskip

Observe that every stable module is a brick, but the converse is not true in general.
Thus, via Conjecture \ref{Conj:CKW}, one can pose a stronger version of the $2nd$ bBT conjecture, as follows.

\begin{conjecture}[Stable Second brick-Brauer-Thrall Conjecture -- Stable 2nd bBT]\label{stable 2nd bBT}
If $A$ is brick-infinite, then there exists some $\theta \in  K_0(\proj A)$ such that there are infinitely many $\theta$-stable modules of the same dimension.
\end{conjecture}

Over tame algebras, one can use some standard geometric arguments to show that the above conjecture is in fact equivalent to the $2nd$ bBT conjecture. More generally, the \emph{Stable} $2nd$ bBT conjecture is a combination of the $2nd$ bBT conjecture and the \emph{Stable brick} conjecture. We remark that Conjecture \ref{stable 2nd bBT} has already appeared in \cite{Pf2}, where the author considers different variations of the $2nd$ bBT conjecture and their connection with Demonet's conjecture. In particular, it is proved (see \cite[Proposition 5.6]{Pf2}) that the \emph{Stable} $2nd$ bBT conjecture implies the \emph{Demonet's} Conjecture. 
Meanwhile, because each stable module is a brick, evidently the \emph{Stable} $2nd$ bBT conjecture implies the \emph{Stable brick} conjecture and the $2nd$ bBT conjecture, hence, it also implies the \emph{Semibrick} conjecture (see Remark \ref{Rem:on semibrick conjecture}).

\medskip

Finally, we remark that another stronger version of the $2nd$ bBT conjecture is posed in terms of the generic bricks \cite[Conjecture 4.1]{MP1}. Because we will treat the ``generic brick conjecture" in our future work, we do not list it here. 

\subsection{Generalized standard components}\label{Subsection: Generalized standard components}
As before, $A$ is always assumed to be a finite dimensional algebra over an algebraically closed field, and $\Gamma_A$ denotes the Auslander-Reiten quiver of $A$. Here we briefly recall some necessary tools and terminology needed for our purposes. For the rudimentary materials on the Auslander-Reiten quivers, see \cite[Chapter IV.4]{ASS}. For the more advanced results appearing below, we refer to \cite{Sk1, Sk2, MS1, MS2, MS3}, and the references therein.

\medskip

Recall that a component $\mathcal{C}$ of $\Gamma_A$ is said to be \emph{regular} if $\mathcal{C}$ contains neither a projective nor an injective module. 
A regular component $\mathcal{C}$ is a called a \emph{stable tube} of rank $r$ if it is of the form $\mathbb{Z}A_{\infty}/\langle \tau^r \rangle$, for some $r\in \mathbb{Z}_{>0}$, where $\tau$ is the translation in $\mathbb{Z}A_{\infty}$.
From \cite{Zh}, it is known that a regular component $\mathcal{C}$ is a stable tube if and only if it contains an oriented cycle (i.e, a path $X_0 \rightarrow X_1 \rightarrow \cdots \rightarrow X_{m-1} \rightarrow X_m$ in $\mathcal{C}$ with $X_0=X_m$). 
For each component $\mathcal{C}$, the cyclic part of $\mathcal{C}$ is the translation subquiver obtained by removing from $\mathcal{C}$ all acyclic
vertices and the arrows attached to them. In particular, $\mathcal{C}$ is called \emph{almost cyclic} if the cyclic part of $\mathcal{C}$ is cofinite in $\mathcal{C}$.
We recall that a component $\mathcal{C}$ of $\Gamma_A$ is said to be \emph{preprojective} if
$\mathcal{C}$ contains no oriented cycle and each module in $\mathcal{C}$ belongs to the $\tau_A$-orbit of a projective module in $\ind(A)$. The \emph{preinjective} components are defined dually.

\medskip

By $\rad(\modu A)$ we denote the radical of $\modu A$, that is, the ideal in $\modu A$ generated by all non-invertible morphisms between modules in $\ind(A)$. 
Then, $\rad^{\infty}(\modu A)$ is defined as the intersection of all $\rad^m(\modu A)$, for all $m \in \mathbb{Z}_{>0}$. 
It is known that $A$ is representation-finite if and only if $\rad^{\infty}(\modu A)=0$. 
Following \cite{Sk2}, a component $\mathcal{C}$ of $\Gamma_A$ is called \emph{generalized standard} if $\rad^{\infty}(X,Y)=0$, for all $X$ and $Y$ in $\mathcal{C}$.
This generalizes the more technical notion of standard components and obviously the Auslander-Reiten quivers of each representation-finite algebra is generalized standard. 
In fact, as shown in \cite{Li3}, each standard component of an Auslander-Reiten quiver of a finite-dimensional algebra over an algebraically closed field is generalized standard. However, the converse is not true; there exist non-standard Auslander-Reiten quivers for some representation-finite algebras over algebraically closed fields of characteristic $2$ (see \cite{Re}).

\medskip

Below, we collect some key tools and make some observations that are freely used in Section \ref{Section: Bricks and generalized standard components}. 
Let us recall that a component $\mathcal{C}$ is said to be \emph{almost periodic} if almost all $\tau_A$-orbits in $\mathcal{C}$ are periodic.
We also observe that $\mathcal{C}$ is a generalized standard component of $A$ if and only if $\mathcal{C}$ is a generalized standard component of $A/\ann(\mathcal{C})$. Here, $\ann(\mathcal{C})$ denotes the annihilator of $\mathcal{C}$. 

\medskip

\begin{theorem}\cite{Sk2}\label{Thm: on gen. standard components}
Let $\mathcal{C}$ be a generalized standard component of $\Gamma_A$. Then, $\mathcal{C}$ is almost periodic. Moreover,
\begin{enumerate}
    \item if $\mathcal{C}$ is regular, then $\mathcal{C}$ is either a stable tube, or else $\mathcal{C}=\mathbb{Z}\Delta$, where $\Delta$ is a finite acyclic quiver \cite[Corollary 2.5]{Sk2}.
    \item if $\mathcal{C}$ is regular and has no $\tau$-periodic module, then $A/\ann(\mathcal{C})$ is a wild tilted algebra \cite[Corollary 3.3]{Sk2}.
    
\end{enumerate}
\end{theorem}

The above facts imply that all regular generalized standard components of a tame algebra must be stable tubes.
In fact, for arbitrary algebras, the next theorem shows that almost all generalized standard components are stable tubes.

\begin{theorem}\label{Thm Sk1: almost all gen. stand. comp. are tubes}
Let $A$ be an algebra of rank $n$.
\begin{enumerate}
    \item The number of generalized standard components in $\Gamma_A$ which are not stable tube is finite \cite[Theorem 3.6]{Sk2}. 
    \item Let $\mathcal{C}$ be a regular generalized standard component in $\Gamma_A$. If $A$ is not strictly wild, then $\mathcal{C}$ is a stable tube of rank $\leq n$ \cite[Corollaries 3.9 $\&$ 5.11]{Sk2}.
\end{enumerate}
    
\end{theorem}
A stable tube in $\Gamma_A$ is known to be generalized standard if and only if its mouth consists of pairwise Hom-orthogonal bricks (\cite[Lemma 1.3]{Sk1}). Moreover, a regular component $\mathcal{C}$ of $\Gamma_A$ is a generalized standard stable tube if and only if $\mathcal{C}$ contains a $\tau_A$-orbit consisting of Hom-orthogonal bricks (\cite[Corollary 5.5]{Sk2}).

\medskip

Starting from an algebra $B$ and a family of infinite generalized standard components in $\Gamma_B$, as in \cite{MS1} one can iteratively apply a sequence of 10 admissible operations to get new translation quivers.
For each admissible operation on the translation quiver, there is a corresponding admissible operation on the algebra under consideration (for these operations on translation quivers and algebras, see \cite[Section 2]{MS1}).
A connected translation quiver $\mathcal{C}$ is called a \emph{generalized multicoil} if it can be obtained from a finite family $\{\mathcal{T}_1,\ldots, \mathcal{T}_r\}$ of stable tubes via an iterated application of the admissible translation operations. 
We remark that for any generalized multicoil $\mathcal{C}$, the cyclic part of $\mathcal{C}$ is infinite, connected and cofinite in $\mathcal{C}$. Thus, any generalized multicoil is a connected almost cyclic translation quiver.

\medskip

Among generalized multicoils, a special attention is given to those obtained from pairwise Hom-orthogonal tubes. In particular, let $B$ be an algebra and $\{\mathcal{T}_i\}$ a family of pairwise Hom-orthogonal generalized standard stable tubes in $\Gamma_B$. Then, an algebra $\Lambda$ is said to be a \emph{generalized multicoil enlargement of $B$} using modules from $\{\mathcal{T}_i\}$ if $\Lambda$ is obtained from $B$ via an iteration of the admissible algebra operations which correspond to those admissible translation operations performed either on stable tubes of $\{\mathcal{T}_i\}$, or on the generalized multicoils obtained from stable tubes of $\{\mathcal{T}_i\}$ by means of operations done so far.
The following theorem describes the connection between the generalized multicoils, as translation quivers, and the generalized multicoil enlargements, as algebras.

\begin{theorem}\cite[Theorem 3.1]{MS2}
Let $A$ be an algebra, $\mathcal{C}$ be a component of $\Gamma_A$, and set $\Lambda=A/ \ann(\mathcal{C})$. Then the following statements are equivalent:

\begin{enumerate}
    \item $\mathcal{C}$ is a generalized standard component and a generalized multicoil.
    \item $\Lambda$ is a generalized multicoil enlargement of an algebra $B$ using modules from a generalized standard family $\{\mathcal{T}_i\}$ of stable tubes of $\Gamma_B$, and $\mathcal{C}$ is the generalized standard multicoil obtained from $\{\mathcal{T}_i\}$ by the admissible operations leading from $B$ to $\Lambda$.
\end{enumerate}

\end{theorem}

\medskip

As shown in \cite[Theorem 1.2]{MS3}, for an infinite generalized standard component $\mathcal{C}$ in $\Gamma_A$, one can consider three quotient algebras of $A$, denoted by $A^{(lt)}_{\mathcal{C}}$, $A^{(rt)}_{\mathcal{C}}$, and $A^{(c)}_{\mathcal{C}}$, respectively called the \emph{left tilted} algebra of $\mathcal{C}$, the \emph{right tilted} algebra of $\mathcal{C}$, and the \emph{coherent} algebra of $\mathcal{C}$. In particular, each one of $A^{(lt)}_{\mathcal{C}}$ and $A^{(rt)}_{\mathcal{C}}$ is a finite product of some tilted algebras. Moreover, $A^{(c)}_{\mathcal{C}}=A^{(c)}_{1}\times \cdots \times A^{(c)}_{p}$, where, for each $1\leq i \leq p$, the algebra $A^{(c)}_{i}$ is a generalized multicoil enlargement of an algebra $B^{(c)}_i$ with a faithful family of pairwise orthogonal generalized standard stable tubes in $\Gamma_{B^{(c)}_i}$ (for details, see \cite[Sections 2 $\&$ 3]{MS3}). As discussed in Section \ref{Section: Bricks and generalized standard components}, this important result allows us to employ several reduction techniques in the study of bricks over those algebras which admit a generalized standard component.

\section{Hom-orthogonal modules and brick-finiteness}\label{Section: Hom-orthogonal modules and brick-finiteness}
Following the same assumptions and notations introduced before, in the following $A$ always denotes a basic, connected, unital associative algebra of rank $n$, over an algebraically closed field $k$.
Moreover, for a collection of modules $\mathcal{X}$ in $\modu A$, by $\supp(\mathcal{X})$ we denote the support of $\mathcal{X}$, that is, the set of simple modules in $\modu A$, considered up to isomorphism, which appear as the composition factor of some $X \in \mathcal{X}$. For every such $\mathcal{X}$, we obviously have $\#\supp(\mathcal{X})\leq n$.

\medskip

Our first result gives a small upper bound on the size of Hom-orthogonal sets of modules over any brick-finite algebra. This also implies part $(1)$ of Theorem A.

\begin{proposition}
\label{Prop:brick-fin and Hom-orthogonality}
Let $A$ be a brick-finite algebra of rank $n$, and $\mathcal{X}$ be a collection of modules in $\modu A$. If $\mathcal{X}$ is Hom-orthogonal, then $\mathcal{X}$ is of size of at most $\#\supp(\mathcal{X})$. In particular, each subset $\{X_1,\dots, X_{n+1}\}$ of $\ind(A)$ is not Hom-orthogonal.  
\end{proposition}
\begin{proof}
We first observe that, for an arbitrary algebra $A$ and every Hom-orthogonal set of modules $\{X_i\}_{i\in I}$ in $\modu A$, we can construct a Hom-orthogonal set of bricks $\{B_i\}_{i\in I}$. In particular, for each $i \in I$, put $B_i:=\Image(f_i)$, where $f_i \in \End_A(X_i)$ is a nonzero endomorphism with $\Image(f_i)$ of minimal dimension. Clearly, for distinct $i,j \in I$, the fact that $\Hom(X_i, X_j)=0$ implies that $\Hom(B_i, B_j)=0$. This, in particular, generates a semibrick of size $\#I$.

Under the assumption that $A$ is brick-finite, every wide subcategory in $\modu A$ is left-finite, and therefore every semibrick brick in $\modu A$ is of size at most $n$ (see \cite[Corollary 2.10]{As}).
Hence, every Hom-orthogonal set in $\modu A$ contains at most $n$ nonzero (non-isomorphic) modules.
To finish the proof, observe that $\mathcal{X}$ can be viewed as a Hom-orthogonal set of modules in $\modu A/J$, where $J$ is the annihilator of $\mathcal{X}$ (i.e, $J$ is the largest ideal in $A$ such that $JX=0$, for every $X \in \mathcal{X}$). Obviously $A/J$ is a brick-finite algebra whose rank is $\#\supp(\mathcal{X})$. The above argument implies that if $\mathcal{X}$ is Hom-orthogonal, then its size is at most $\#\supp(\mathcal{X})$.
\end{proof}

\begin{remark}\label{Rem: brick-finite has no more than n orthogonal modules}
Observe that if $A$ is a (connected) algebra of rank $n$, we have $|\brick(A)|\geq n$. In fact, $|\brick(A)|=n$ if and only if $n=1$, which is the case, if and only if $A$ is a local algebra. This is equivalent to saying that, for each pair of non-isomorphic $A$-modules $X_1$ and $X_2$, we have $\Hom_A(X_1,X_2)\neq 0$. Hence, if $A$ is a local algebra, no subset $\{X_1, X_2\}$ of $\ind(A)$ is Hom-orthogonal.
From this viewpoint, Proposition \ref{Prop:brick-fin and Hom-orthogonality} gives a characterization of brick-finite algebras as a generalization of local algebras: If $A$ is a brick-finite algebra of rank $n$, no subset $\{X_1,\dots, X_{n+1}\}$ of $\ind(A)$ is Hom-orthogonal.
\end{remark} 

Some immediate implications of Proposition \ref{Prop:brick-fin and Hom-orthogonality} is collected in the next corollary, which we state without a separate proof.

\begin{corollary}\label{cor:infinite family of orthogonal modules}
Let $A$ be an algebra of rank $n$. Then, 
\begin{enumerate}
    \item If there exists a Hom-orthogonal set of size $n+1$, then $A$ is brick-infinite.
    \item $A$ admits an infinite set of Hom-orthogonal modules in $\modu A$ if and only if $A$ admits an infinite set of Hom-orthogonal bricks.
\end{enumerate}
\end{corollary}

It is expected that the converse of part $(1)$ of the above corollary is true in general (see Remark \ref{Rem:on semibrick conjecture}).
Moreover, if the infinite Hom-orthogonal set in part $(2)$ has bounded dimensions, then the same holds for the infinite set of bricks, which would yield the $2nd$ bBT conjecture (Conjecture \ref{Conj: 2ndbBT}).
Note that, if $A$ is brick-infinite, then by the $2nd$ BT conjecture (now theorem), there exists an ascending sequence of integers $0<d_1 < d_2 < \cdots$ such that every $\ind(A,d_i)$ is an infinite set. Thus, to prove the $2nd$ bBT conjecture for a brick-infinite algebra $A$, it suffices to show that one of such infinite families $\ind(A,d_i)$ contains an infinite Hom-orthogonal subset. In Section \ref{Section: Orthogonal orbits and irreducible components} we prove a stronger version of this fact.
\medskip

Through some geometric arguments in representation varieties, in the next proposition we obtain a more refined version of some known results on brick components (compare \cite[Corollary 1.6]{G+}). 
In that paper, the authors have shown that if $\mathcal{Z}$ is a brick component without a dense orbit, then there is an open set in $\mathcal{Z}\times \mathcal{Z}$ consisting of pairs of orthogonal bricks. We note that projecting such an open set down onto $\mathcal{Z}$ does not necessarily yield a set of Hom-orthogonal bricks. That being the case, to obtain the desired set of Hom-orthogonal bricks in $\mathcal{Z}$, some further observations are made in the proof of the following proposition. 
Below, for $X$ in $\modu A$, we set ${X}^{\perp_0}:=\{Y\in \modu A\,|\, \Hom_A(X,Y)=0\}$. We dually define $\,^{\perp_0}{X}$.
Note that, together with Corollary \ref{cor:infinite family of orthogonal modules}, the following proposition immediately implies part $(2)$ of Theorem A.  

\begin{proposition}\label{Prop:2nd bBT gives an infinite family of orthogonal bricks}
For algebra $A$, and each $d \in \mathbb{Z}_{>0}$, the following are equivalent:

\begin{enumerate}
    \item $A$  admits an infinite family of (non-isomorphic) bricks of dimension $d$.
    \item $\brick(A,d)$ contains an infinite set of Hom-orthogonal bricks.
\end{enumerate}
\end{proposition}
\begin{proof}
We only need to show the implication $(1) \rightarrow (2)$. In particular, our assumption implies that the variety $\modu (A,d)$ contains a brick component $\mathcal{Z}$ with no dense orbit. Hence, from \cite[Corollary 1.6]{G+}, we know that there is a non-empty open subset $O$ of $\mathcal{Z}\times \mathcal{Z}$ such that, for each pair $(M,N)$ in $O$, we have $\Hom_A(M,N)=0$. 
To this open subset $O \subseteq \mathcal{Z}\times \mathcal{Z}$, we associate the subset $U:=\pi_1(O) \cap \pi_2(O)$ in  $\mathcal{Z}$, where $\pi_1, \pi_2: \mathcal{Z}\times \mathcal{Z} \to \mathcal{Z}$ are the two natural projection maps. These projection maps are known to be open, thus $U$ is a non-empty open subset of $\mathcal{Z}$. 

Now, for each $X_1$ in $U$, observe that ${X_1}^{\perp_0}$ and $\,^{\perp_0}{X_1}$ are non-empty open subsets of $U$. Define $U_1:={X_1}^{\perp_0} \cap \,^{\perp_0}{X_1}$, which is non-empty. Obviously $X_1 \notin U_1$, and therefore $U\supsetneq U_1$. Then, for any $X_2$ in $U_1$, we similarly have that ${X_2}^{\perp_0}$ and $\,^{\perp_0}{X_2}$ are non-empty open subsets of $U_1$, and define $U_2:={X_2}^{\perp_0} \cap \,^{\perp_0}{X_2}$.
Repeating this process, we obtain an infinite chain $U \supsetneq U_1 \supsetneq U_2 \supsetneq \cdots$ of non-empty open subsets of $\mathcal{Z}$, which gives us an infinite set $\{X_1, X_2, \cdots \}$ of pairwise Hom-orthogonal modules in $\mathcal{Z}$.
\end{proof}

Before we make some observations on the preceding results, let us remark that Proposition \ref{Prop:2nd bBT gives an infinite family of orthogonal bricks}, together with Proposition \ref{Prop:brick-fin and Hom-orthogonality} and Corollary \ref{cor:infinite family of orthogonal modules}, give a complete proof of Theorem A in Section \ref{Section: Introduction}.

\begin{remark}\label{Rem:on semibrick conjecture}(\emph{Strong Semibrick Conjecture})
The semibrick conjecture asserts that an algebra $A$ is brick-infinite if and only if $A$ has an infinite semibrick (Conjecture \ref{Conj:Semibrick Conjecture}). Meanwhile, Proposition \ref{Prop:2nd bBT gives an infinite family of orthogonal bricks} implies that if the $2nd$ bBT conjecture (Conjecture \ref{Conj: 2ndbBT}) is true, then the semibrick conjecture is immediate.
This leads to a sharper version of the semibrick conjecture, asserting that the following are equivalent:
\begin{enumerate}
    \item $A$ is brick infinite;
    \item $A$ admits an infinite semibrick consisting of bricks of the same dimension.
\end{enumerate}
Although $(2)\rightarrow (1)$ is obvious, the implication $(1)\rightarrow (2)$ is highly non-trivial and is equivalent to the $2nd$ bBT conjecture. That is, the $2nd$ bBT conjecture can be seen as a stronger version of the semibrick conjecrure. 
We note that our results in \cite{MP1, MP2} immediately imply this stronger semibrick conjecture for some families of algebras (also see Corollary \ref{Cor:reduction to min-brick-inf}).
\end{remark}

By a systematic treatment of the $2nd$ bBT conjecture via $\tau$-tilting theory, in \cite{MP3} we reduced this conjecture to a certain family of algebras. Before we state the connections between our new results and this reduction, recall that we say $A$ is \emph{minimal brick-infinite} algebra if $A$ is brick-infinite, but each proper quotient of $A$ is brick-finite. 
We note that if $B$ is a brick-infinite quotient of $A$ and satisfies one of the conjectures from Section \ref{Subsection:bBT Conjs}, then that conjecture immediately holds for $A$.
In particular, to prove the (stable) $2nd$ bBT conjecture for a brick-infinite algebra, it obviously suffices to verify that over one of its minimal brick-infinite quotients. 
Moreover, \cite[Theorem 1.4]{MP3} gives a non-trivial reduction on this open problem: To settle the $2nd$ bBT conjecture in full generality, it is sufficient to verify it over those minimal brick-infinite algebras for which almost all bricks are faithful.
Such reductive arguments can be similarly employed in the study of the other brick-Brauer-Thrall conjectures. For instance, in \cite{MP1}, we explicitly described the minimal brick-infinite biserial algebras. From that, one can verify all the conjectures in Section \ref{Subsection:bBT Conjs} for arbitrary biserial algebras. 
On the other hand, we can strengthen our earlier results on some minimal brick-infinite algebras.

\begin{corollary}\label{Cor:reduction to min-brick-inf}
Let $A$ be a minimal brick-infinite algebra that admits infinitely many unfaithful bricks. Then, $A$ admits an infinite semibrick consisting of bricks of the same dimension.
\end{corollary}
\begin{proof}
By \cite[Theorem 5.4]{MP3}, $A$ admits an infinite family of bricks of the same dimension. Now, the desired result follows from Proposition \ref{Prop:2nd bBT gives an infinite family of orthogonal bricks}.
\end{proof}

From Remark \ref{Rem:on semibrick conjecture} and the preceding corollary, it is immediate that the semibrick conjecture holds for any algebra that admits a minimal brick-infinite quotient of the type described in Corollary \ref{Cor:reduction to min-brick-inf}.

\medskip

The connection between the Hom-orthogonal set of bricks and wide subcategories has been studied in \cite{As} and \cite{En1}. Our results on Hom-orthogonal modules can be further investigated through that lens.

\begin{proposition}\label{Cor:brick-fin simple hom-orthogonal}
For each brick-finite algebra $A$ of rank $n$, the only semibrick of size $n$ is comprised of simple $A$-modules.
\end{proposition}
\begin{proof}
Let $\mathcal{X}:=\{X_1, \cdots, X_n\}$ be a semibrick. Since $A$ is brick-finite, then $\mathcal{X}$ is a left-finite semibrick, hence it corresponds to the set of minimal co-extending bricks of a functorially finite torsion class (see \cite[Section 2.2]{As}). The only functorially finite torsion class having $n$ minimal co-extending bricks is $\modu A$, and the corresponding minimal co-extending bricks is the set of simple $A$-modules. Thus, $\mathcal{X}$ has to be the set of simple $A$-modules.
\end{proof}

We end this section by the following remark, which shows that the $2nd$ bBT conjecture implies a new characterization of brick-finite algebras in terms of wide subcategories. Although we do not make a detailed study of wide subcategories in this work, we state the following result to highlight some applications of the brick-Brauer-Thrall conjectures to some other problems.

\begin{remark}
Should Conjecture \ref{Conj: 2ndbBT} hold for an algebra $A$, then $A$ is brick-finite if and only if every wide subcategory in $\modu A$ is functorially finite. 
\end{remark}

\begin{proof}
We first recall that a wide subcategory $\mathcal{W}$ in $\modu A$ is functorially finite if and only if $\mathcal{W}$ is equivalent to $\modu B$, for a finite dimensional algebra $B$ (for example, see \cite[Proposition 4.12]{En2}).
In particular, if $A$ is brick-finite, then each wide subcategory is of this form, that is, every wide subcategory $\mathcal{W}$ in $\modu A$ is functorially finite (see \cite[Corollary 3.11]{MS}).
Observe that if $A$ is brick-infinite, Conjecture \ref{Conj: 2ndbBT} implies that there is an infinite family of bricks of the same dimension, say $\{X_i\}_{i \in \mathbb Z}$ in $\brick(A)$. Moreover, Proposition \ref{Prop:2nd bBT gives an infinite family of orthogonal bricks} yields an infinite semibrick as a subset of $\{X_i\}_{i \in \mathbb Z}$. Now, if we consider the wide subcategory $\mathcal{W}$ generated by this infinite semibrick, then we obtain a wide subcategory $\mathcal{W}$ in $\modu A$ which has infinitely many simple objects. Evidently, $\mathcal{W}$ cannot be equivalent to the module category of a finite dimensional algebra, that is, $\mathcal{W}$ is not functorially finite. 
\end{proof}

\section{Orthogonal orbits and irreducible components}\label{Section: Orthogonal orbits and irreducible components}
In the following, $A$ is always assumed to be a basic, connected finite dimensional algebra of rank $n$ over an algebraically closed field $k$.  
For an irreducible component $\Z \in \Irr(A)$, two orbits $O_X$ and $O_Y$ in $\Z$ are said to be orthogonal if $X$ and $Y$ are Hom-orthogonal modules.
In this section we focus on the behaviour of the Hom-orthogonal modules in the irreducible components of $A$. 
Through this geometric approach, we obtain a surprisingly elementary equivalent condition for the existence of an infinite family of bricks of the same dimension (Theorem \ref{Thm:1-side vanishing Hom of two modules} and Corollary \ref{Cor: geometric version of local algebras}). In particular, we prove Proposition B from Section \ref{Section: Introduction}. We also give a characterization of brick-finite algebras as a generalization of the local algebras.

\medskip

We begin with a handy lemma, which should be known to experts, but we have not found it explicitly written in the literature. The proof uses some standard geometric arguments with upper-semicontinuity.
\begin{lemma} \label{No orhtogonal modules when a dense orbit}
If $\Z \in \Irr(A)$ has a dense orbit, then it contains no pair of modules $X$ and $Y$ with $\Hom_A(X,Y)=0$.
\end{lemma}

\begin{proof} 
Let $Z$ be such that $\overline{O_Z} = \Z$. Assume that the conclusion is false.  Then $\Z \times \Z$ contains a pair $(X,Y)$ with $\Hom_A(X,Y)=0$. By upper semicontinuity, there is an open (non-empty) set $\mathcal{O}$ in $\Z \times \Z$ for which $(M,N) \in \mathcal{O}$ implies $\Hom_A(M,N)=0$. But $O_Z \times O_Z$ is also an open set of $\Z \times \Z$. Since the latter is irreducible, the open sets $O_Z \times O_Z$ and $\mathcal{O}$ intersect, which gives the absurdity $\Hom_A(Z,Z)=0$.  
\end{proof}

The preceding lemma leads to the following result, using ideas similar to the proof of Proposition \ref{Prop:2nd bBT gives an infinite family of orthogonal bricks}. In particular, we verify the Strong Semibrick Conjecture for the algebras under consideration (see Remark \ref{Rem:on semibrick conjecture}).

\begin{proposition}
    \label{Thm:1-side vanishing Hom of two modules}
Let $\Z \in \Irr(A)$ contain a pair of non-isomorphic modules $X$ and $Y$ with $\Hom_A(X,Y)=0$.
Then, $\Z$ contains an infinite family of Hom-orthogonal modules. That being the case, $A$ admits an infinite semibrick consisting of bricks of the same dimension.
\end{proposition}

\begin{proof}
From Lemma \ref{No orhtogonal modules when a dense orbit}, it follows that $\Z$ does not contain a dense orbit. For a module $Z$ in $\Z$, note that the upper semicontinuity of ${\rm dim}_k\Hom_A(Z,-)$ gives an (possibly empty) open subset $Z^{\perp_0}$ in $\Z$, such that $\Hom_A(Z,N) = 0$, for all $N \in Z^{\perp_0}$. Similarly, there is an (possibly empty) open subet $\,^{\perp_0} Z$ in $\Z$ such that $\Hom_A(N,Z)=0$ for all $N \in \,^{\perp_0} Z$.

Now, let us fix $M$ in the non-empty subset $X^{\perp_0}$. The open set $\,^{\perp_0} M$ is again non-empty. Taking $L \in \,^{\perp_0} M$, we similarly get a non-empty open set $L^{\perp_0}$. Take $M_1$ in the non-empty intersection $\,^{\perp_0} M \cap L^{\perp_0}$, so that both $M_1^{\perp_0}$ and $\,^{\perp_0} M_1$ are non-empty. From this, we inductively define $M_i$ as being a module in the non-empty open set $O_i:=\bigcap_{j=1}^{i-1}\left(M_{j}^{\perp_0} \cap \,^{\perp_0} M_{j}\right)$, which consists of an infinite union of orbits. In particular, we get an infinite descending chain $O_1\supsetneq O_2 \supsetneq O_3 \supsetneq \cdots$ of open sets, each of which contains infinitely many orbits. 
In this way, we have defined an infinite family $\{M_i\}_{i \in \mathbb{Z}_{\geq 1}}$ of Hom-orthogonal modules in $\Z$.

For the last assertion on the existence of the infinite semibrick with bricks of the same dimension, we first note that Corollary \ref{cor:infinite family of orthogonal modules} implies that the infinite Hom-orthogonal set $\{M_i\}_{i \in \mathbb{Z}_{\geq 1}}$ induces an infinite semibrick $\mathcal{S}:=\{B_i\}_{i \in \mathbb{Z}_{\geq 1}}$. From the construction of the $B_i$, we know that $B_i$ is a submodule (and also a quotient module) of $M_i$.
In particular, since the $M_i$ belong to $\Z$, and therefore they are of the same dimension, we have an infinite subset of $\mathcal{S}$ consisting of bricks of the same dimension. This finishes the proof.
\end{proof}

The previous proposition has some interesting implication. In the following corollary, we give a new characterization of the brick-finite algebras as a generalization of the local algebras from the viewpoint of geometry. Since the result follows from the above theorem, we eliminate the proof. However, to motivate the assertion, let us recall that if $A$ is a local algebra, then for every pair of nonzero $A$-modules $X$ and $Y$, we have $\Hom_A(X,Y)\neq 0$.

\begin{corollary}\label{Cor: geometric version of local algebras}
If algebra $A$ is brick-finite, then for each $\Z \in \Irr(A)$, and every pair of modules $X$ and $Y$ in $\Z$, we have $\Hom_A(X,Y) \ne 0$.
\end{corollary}

Note that the $2nd$ bBT conjecture (Conjecture \ref{Conj: 2ndbBT}) implies that the converse of the above corollary should also hold, that is, $A$ is brick-finite if and only if whenever $\Hom_A(X,Y) = 0$, then $X$ and $Y$ belong to different components in $\Irr(A)$.

\medskip

As another consequence of Proposition \ref{Thm:1-side vanishing Hom of two modules}, and in light of Proposition \ref{Prop:2nd bBT gives an infinite family of orthogonal bricks}, we obtain a new equivalent condition for the $2nd$ bBT conjecture.
To put this result in perspective, recall that the celebrated Second Brauer-Thrall conjecture (now theorem) asserts that $A$ is a representation-infinite algebra if and only if there are infinitely many $d \in \mathbb{Z}_{>0}$ for which $\ind(A,d)$ is an infinite set. 
We observe that, to verify the $2nd$ bBT conjecture for $A$, it suffices to show that for some $d \in \mathbb{Z}_{>0}$ and infinitely many pairs $X$ and $Y$ in $\ind(A,d)$, we have $\Hom_A(X,Y)=0$. In other words, the $2nd$ bBT conjecture asserts $A$ is brick-infinite if and only if there are infinitely many Hom-orthogonal $A$-modules of the same dimension.

\section{Bricks and generalized standard components}\label{Section: Bricks and generalized standard components} 

As briefly recalled in Section \ref{Section: Preliminaries and Background}, among components of the Auslander-Reiten quivers, the generalized standard components feature some remarkable properties and have been crucial in various classification problems (see \cite{Sk1, Sk2, Li1, Li2, Li3, MS1, MS2, MS3} and the references therein). For instance, preprojective components are known to be generalized standard, and those algebras that admit a preprojective component are historically important and received a lot of attention in the study of the Second Brauer-Thrall conjecture -- now theorem (for details, see \cite{HV} and \cite[Section 6]{Bo}).
Here we consider the algebras that admit an arbitrary generalized standard component and treat them from the viewpoint of some modern analogues of the Second Brauer-Thrall conjecture. In fact, we show that all conjectures from Section \ref{Subsection:bBT Conjs} hold for any algebra that admits a generalized standard component.
This particularly extends some results of \cite{CKW}, where the authors treated some geometric properties of the algebras with a preprojective component. 
As a consequence, we obtain Theorem C and Corollary D stated in Section \ref{Section: Introduction}.

\medskip

Before stating our results, in the next paragraph, we lay out our methodology and make some observations to put our setting in perspective with some similar research conducted recently. Throughout, we freely use the terminology and results from Section \ref{Subsection: Generalized standard components} and the references provided there. 

\medskip

As mentioned previously (see Corollary \ref{Cor:reduction to min-brick-inf} and the paragraph before that), a useful approach to the study of the brick-Brauer-Thrall conjectures is the reduction to quotient algebras. Such reductions are extensively used while showing some of the brick-Brauer-Thrall conjectures for some families (see \cite{CKW} and \cite{MP1, MP2, MP3}).
Here we particularly highlight two important points. First, we recall that if $B$ is a quotient algebra of $A$, then $\modu B$ is a full subcategory of $\modu A$, thus $\brick(B,d')\subseteq \brick(A,d)$, for each $d\in \mathbb{Z}_{>0}$.
Second, we note that if $\Gamma_A$ has an infinite preprojective component, then there is a quotient algebra of $A$ that is tame concealed (see \cite[Chapter XIV]{SS}). Meanwhile, some fundamental results on the (tame) concealed algebras (see \cite[Chapter XII]{SS}), together with some geometric facts on bricks and stable modules, imply that all brick-Brauer-Thrall conjectures from Section \ref{Subsection:bBT Conjs} hold for concealed algebras.
Thus, if $\Gamma_A$ has a preprojective component, the stable brick conjecture (Conjecture \ref{Conj:CKW}) is verified for $A$ (compare \cite[Theorem 3.5 $\&$ Conjecture 5.4]{CKW}), from which the Second brick-Brauer-Thrall conjecture (Conjecture \ref{Conj: 2ndbBT}) will follow for $A$ (compare \cite[Section 6]{Mo1}), hence Proposition \ref{Prop:2nd bBT gives an infinite family of orthogonal bricks} implies that the semibrick conjecture holds for $A$ (Remark \ref{Rem:on semibrick conjecture}).
In contrast, if $\Gamma_A$ has an arbitrary (infinite) generalized standard component, $A$ does not necessarily admit a concealed quotient algebra. Thus, we employ a range of geometric and algebraic tools while treating different quotients of algebras under consideration.

\medskip

Our first proposition in this section treats the problem of existence of stable bricks in arbitrary generalized standard stable tubes, leading to some new results which are also useful in our reductive arguments.
Following our notation from Section \ref{Subsection: Rep. varieties and components}, for any module $X \in \modu A$, we consider the weight $\theta_X \in K_0(\proj A)$. More precisely, for each $M \in \modu A$ and the corresponding element $[M]=\underline{\dim}M$ in the Grothendieck group $K_0(\modu A)$, we have
$$\theta_X([M]):=\dim_k\Hom_A(X,M)-\dim_k\Hom_k(M,\tau X).$$

\medskip

Observe that, if $X$ is a homogeneous brick, then $X$ is $\theta_X$-stable (see Lemma \ref{Lem: Homogeneous brick is stable}). Hence, for homogeneous generalized standard stable tubes, the assertion of the following proposition follows from some earlier observations. However, a direct generalization of the above argument does not seem to work for arbitrary generalized standard stable tubes. Consequently, we construct a stability parameter using the direct sum of some modules along a co-ray ending at the mouth of the tube. This is the core of the following proof.

\begin{proposition}\label{Prop:standard tubes}
Let $\Gamma_A$ contain a generalized standard stable tube. Then, there exists a stability parameter $\theta \in K_0(\proj A)$ for which $A$ admits an infinite family of $\theta$-stable bricks of the same dimension.
\end{proposition}
\begin{proof}
Let $\mathcal{T}$ be a generalized standard stable tube of rank $r$ in $\Gamma_A$. As shown in \cite[Lemma 1.3]{Sk1}, the mouth of $\mathcal{T}$ consists of pairwise orthogonal bricks.
Consider the quotient algebra $B:=A/\ann(\mathcal{T})$, and observe that $\mathcal{T}$ is a faithful generalized standard stable tube in $\Gamma_B$. Therefore, for all $X \in \mathcal{T}$, we have $\pd_B(X)\leq 1$ and $\id_B(X)\leq 1$ (see \cite[Lemma 5.9]{Sk2}). 

Let $X_1$ be on the mouth of the tube $\mathcal{T}$ and consider the irreducible epimorphisms $X_{i+1} \to X_i$, for $i = 1, \ldots, r-1$. In other words, we take the last $r$ modules on the coray $\cdots \to X_3\to X_2 \to X_1$ in $\mathcal{T}$.
Put $Y := X_1 \oplus \dots \oplus X_r$ and consider the corresponding stability parameter $\theta:=\theta_Y(-)$ in $K_0({\rm proj} B)$. Observe that $\pd_B(Y)\leq 1$ and $\theta$ is well-defined. We claim that $X_r$ is a $\theta$-stable $B$-module.

First, using that $\mathcal{T}$ is a (generalized) standard tube, it is straightforward to check that $\Hom(Y,X_r) = \Hom(X_r, X_r)$ and also $\Hom(X_r, \tau Y) = \Hom(X_r, \tau X_r)$ are both one-dimensional. Therefore, $\theta([X_r])=0$. To prove that $X_r$ is $\theta$-stable, we next need to show that for any non-zero proper submodule $L$ of $X_r$, we have $\theta([L])<0$. 
If $\Hom(Y,L)\neq 0$, then $\Hom(X_i, L)\neq 0$, for some $1 \leq i \leq r$. Since $L$ is a submodule of $X_r$, this evidently implies $\Hom(X_i, X_r)\neq 0$. However, due to our choice of $X_i$'s, the only possibility is $i = r$. Because $X_r$ is a brick, and $L$ is a proper submodule of $X_r$, this is impossible. From this, we get $\Hom(Y,L)=0$, and it only remains to show $\Hom(L, \tau Y)$ is non-zero. We note that all quasi-simple modules in $\mathcal{T}$ appear as the quasi-socle of $\tau X_i$, for some $1 \le i \le r$. In particular, for each module $Z$ on the mouth of $\mathcal{T}$, there exists $1\leq i \leq r$ for which $Z$ is a submodule of $\tau X_i$.
Therefore, if $\Hom(L, \tau Y)=0$, then $\Hom(L,-)$ vanishes on all quasi-simple objects of $\mathcal{T}$. Since every object of $\mathcal{T}$ has a filtration by the quasi-simple modules of the tube, this implies that $\Hom(L,-)$ vanishes on all objects of $\mathcal{T}$. This particularly implies $\Hom(L, X_r)=0$, which is evidently a contradiction. This shows that $\Hom(L, \tau Y)\neq 0$ and because $\Hom(Y,L)=0$, we obtain $\theta([L])<0$. Hence, $X_r$ is $\theta$-stable, as claimed.

Finally, we note that, for any $d\in \mathbb{Z}_{>0}$ and each $\theta \in K_0(\proj B)$, it is well-known that the $d$-dimensional $\theta$-stable $B$-modules form an open subset $\mathcal{O}$ of the variety $\rep(B,d)$. In particular, for $d:=\dim(X_r)$ and $\theta:=\theta_Y$, since $X_r$ is $\theta$-stable, $\mathcal{O}$ is non-empty. On the other hand, $\pd_B(X_r)\leq 1$ and $\Hom_B(X_r, \tau_B X_r) \ne 0$. Hence, Proposition \ref{Prop: bricks with pd=1} implies that $\rep(B,d)$ contains an open subset $\mathcal{O}'$ which is an infinite union of orbits of bricks. Consequently, $\mathcal{O} \cap \mathcal{O}'$ is a non-empty open set which consists of an infinite union of orbits of $\theta$-stable bricks of dimension $d$. 
Since $B$ is a quotient of $A$, we have an induced split epimorphism $K_0({\rm proj} A) \to K_0({\rm proj} B)$ sending $[P]$ to $[P/\ann(\mathcal{T})P]$. Taking any $\theta' \in K_0({\rm proj} A)$ that maps to $\theta$ immediately yields that our infinite family of $\theta$-stable bricks in $\brick(B,d)$ yields an infinite family of $\theta'$-stable bricks in $\brick(A,d)$. This finishes the proof.

\end{proof}

Since almost all generalized standard components are stable tubes (Theorem \ref{Thm Sk1: almost all gen. stand. comp. are tubes}), the above proposition has some nice consequences. As we see later, this particularly has some interesting implications over tame algebras (see Corollary \ref{Cor:equiv. conditions for tame algebras}). 

\medskip

The following theorem is the main result of this section. Before we state this result, let us observe that any algebra that satisfies the condition of Proposition \ref{Prop:standard tubes} is obviously representation-infinite. Meanwhile, the Auslander-Reiten quiver of each representation-finite is generalized standard. In the next theorem, we treat the existence of arbitrary generalized standard components.
This more general setting also allows us to better highlight how our work extends some earlier results on the algebras that admit a preprojective component (compare \cite[Theorem 3.5]{CKW}).

\begin{theorem}\label{Thm: Alg. with generalized standard components}
If $\Gamma_A$ contains a generalized standard component, the following are equivalent:
\begin{enumerate}
    \item $A$ is representation-finite;
    \item for each dimension vector $\underline{d}$, and each irreducible component $\Z \in \Irr(A,\underline{d})$, there is a dense orbit in $\mathcal{Z}$;
    \item $\brick(A,d)$ is finite, for all $d\in \mathbb{Z}_{>0}$;
    \item $A$ is brick-finite;
    \item For each $\theta \in K_0(\proj A)$ and each dimension $d$, there exist only finitely many $\theta$-stable modules of dimension $d$.
\end{enumerate}
\end{theorem}

Before providing a proof, we recall some key tools and notations. As mentioned in Section \ref{Subsection: Generalized standard components}, if $\mathcal{C}$ is an infinite generalized standard component in $\Gamma_A$, by \cite[Theorem 1.2]{MS3}, $A$ admits three quotient algebras $A^{(lt)}_{\mathcal{C}}$, $A^{(rt)}_{\mathcal{C}}$, and $A^{(c)}_{\mathcal{C}}$, respectively called the \emph{left tilted}, the \emph{right tilted}, and the \emph{coherent} algebra of $\mathcal{C}$ (for the definition and properties, see \cite[Section 2 $\&$ 3]{MS3}).
Here we emphasize that $A^{(lt)}_{\mathcal{C}}$ and $A^{(rt)}_{\mathcal{C}}$ are finite (possibly trivial) products of tilted algebras, and $A^{(c)}_{\mathcal{C}}$ is a finite (possibly trivial) product of some generalized multicoil enlargements a family of algebras.
Moreover, if $\mathcal{C}^{(l)}$, $\mathcal{C}^{(r)}$, and $\mathcal{C}^{(c)}$ respectively denote the translation quivers of $A^{(lt)}_{\mathcal{C}}$, $A^{(rt)}_{\mathcal{C}}$, and $A^{(c)}_{\mathcal{C}}$, then $\mathcal{C}^{(l)}\cupdot \mathcal{C}^{(r)}$ contains almost all acyclic elements of $\mathcal{C}$, and $\mathcal{C}^{(c)}$ contains almost all cyclic elements of $\mathcal{C}$. Because $\mathcal{C}$ is an infinite component, at least one of the algebras $A^{(lt)}_{\mathcal{C}}$, $A^{(rt)}_{\mathcal{C}}$, and $A^{(c)}_{\mathcal{C}}$ is non-trivial and representation-infinite.

\medskip

\begin{proof}[Proof of Theorem \ref{Thm: Alg. with generalized standard components}]
We first observe that the implications $(1)\rightarrow (4)\rightarrow (3) \rightarrow (5)$ are evident from the definition. Moreover, $(1)\rightarrow (2)\rightarrow (3)$ hold for arbitrary algebras. 
In particular, if $A$ is representation-finite, each $\Z \in \Irr(A)$ has only finitely many orbits, implying that $c(\Z)=0$. Therefore, $\Z=\overline{O_X}$, for some $X\in\Z$. This shows $O_X$ is dense, and verifies $(1)\rightarrow (2)$. Furthermore, as discussed in Section \ref{Subsection: Rep. varieties and components}, each brick component $\Z \in \Irr(A)$ contains either a unique orbit, or else $\Z$ has infinitely many orbits. In the latter case, $c(\Z)>0$, which contradicts the assumption that all components $\Z \in \Irr(A)$ have dense orbit. Thus, for every $d \in \mathbb{Z}_{>0}$, there are only finitely many bricks of dimension $d$. This shows $(2)\rightarrow (3)$.

Thanks to the above general facts, to complete the sequence of implications and conclude the equivalence of all, we only need to show that the implications $(5)\rightarrow (1)$ holds for any algebra which admits a generalized standard component.
Meanwhile, observe that a component $\mathcal{C}$ in $\Gamma_A$ is finite if and only if $A$ is a representation-finite algebra (see \cite[VII 2.1]{ARS}).
In particular, if $\Gamma_A$ has a finite generalized standard component, we have nothing to prove. Thus, henceforth we assume $\mathcal{C}$ is an infinite generalized standard component in $\Gamma_A$. In this case, we prove that $A$ admits an infinite family of $\theta$-stable bricks of the same dimension, for some $\theta \in K_0(\proj A)$.

As noted in \cite{MS3}, if the infinite generalized standard component $\mathcal{C}$ admits left stable acyclic part, then the left tilted algebra $A^{(lt)}_{\mathcal{C}}$ is nontrivial. That being the case, from the description of $A^{(lt)}_{\mathcal{C}}$ in \cite[Theorem 1.2 (i)]{MS3}, it follows that $A$ has a quotient algebra which is tilted and representation-infinite, so the desired result is immediate, particularly because such an algebra admits an infinite preprojective component, and hence has a tame concealed quotient algebra (compare \cite[Theorem 3.5]{CKW}). A similar argument holds if $\mathcal{C}$ admits right stable acyclic part (see \cite[Theorem 1.2 (ii)]{MS3}).
Hence, we can further assume every (infinite) generalized standard component $\mathcal{C}$ in $\Gamma_A$ is such that both $A^{(lt)}_{\mathcal{C}}$ and $A^{(rt)}_{\mathcal{C}}$ are trivial. This implies that the coherent algebra $A^{(c)}_{\mathcal{C}}$ must be representation-infinite. We show that $A^{(c)}_{\mathcal{C}}$ admits an infinite family of $\theta$-stable bricks of the same dimension, for some stability parameter $\theta$.

By \cite[Theorem 1.2 (iii)]{MS3}, we have $A^{(c)}_{\mathcal{C}}=A^{(c)}_{1}\times \cdots \times A^{(c)}_{p}$, where, for each $1\leq i \leq p$, the algebra $A^{(c)}_{i}$ is a generalized multicoil enlargement of an algebra $B^{(c)}_i$ with a faithful family of pairwise orthogonal generalized standard stable tubes in $\Gamma_{B^{(c)}_i}$.
As mentioned in Section \ref{Subsection: Generalized standard components}, the algebra $A^{(c)}_{i}$ is obtained from $B^{(c)}_i$, via a sequence of admissible algebra operations given in \cite[Section 3]{MS2}. From the aforementioned operations, it follows that $B^{(c)}_i$ is a quotient algebra of $A^{(c)}_{i}$, and therefore a quotient of $A$. Hence, $\modu B^{(c)}_i$ is a full subcategory of $\modu A$.
Thus, to finish the proof, we note that $B^{(c)}_i$ admits a faithful family $\mathcal{T}^{B^{(c)}_i}$ of pairwise orthogonal generalized standard stable tubes. Hence, by Proposition \ref{Prop:standard tubes}, $B^{(c)}_i$ admits an infinite family of $\theta'$-stable bricks of the same dimension, for some stability parameter $\theta' \in K_0(\proj B^{(c)}_i)$.
As argued in the last paragraph of the proof of Proposition \ref{Prop:standard tubes}, this $\theta'$ gives rise to some $\theta \in K_0(\proj A)$ for which we have infinitely many $\theta$-stable bricks of the same dimension. This particularly proves the implication $(5)\rightarrow (1)$, and hence finishes the proof of the theorem.
\end{proof}

\medskip

The above theorem has some interesting applications in the study of the open conjectures from Section \ref{Subsection:bBT Conjs}. In particular, the following corollary shows that, for the algebras that admit a generalized standard component, all of the conjectures listed in Section \ref{Subsection:bBT Conjs} are equivalent, and more importantly, they all hold for any such algebra.

\newpage

\begin{corollary}\label{Cor: all bBT conjectures for gen. stan. comp.}
Let $A$ be such that $\Gamma_A$ contains a generalized standard component. Then the following are equivalent:
\begin{enumerate}
    \item $A$ is brick-infinite;
    \item $A$ admits an infinite family of bricks of the same dimension;
    \item $A$ admits an infinite semibrick;
    \item There is a rational ray outside of the $\tau$-tilting fan of $A$;
    \item For some $\theta \in K_0(\proj A)$, there is an infinite family of $\theta$-stable bricks of the same dimension.
\end{enumerate}
\end{corollary}

\begin{proof}
Thanks to Theorem \ref{Thm: Alg. with generalized standard components}, we only need to show that $(3)$ and $(4)$ are equivalent to any of the other assertions.
Observe that $(3)\rightarrow (1)$ and $(4)\rightarrow (1)$ are evident, and Proposition \ref{Prop:2nd bBT gives an infinite family of orthogonal bricks} implies $(2)\rightarrow (3)$. Hence, we only show $(5)\rightarrow (4)$. 

Assume that we have $\theta \in K_0(\proj A)$ such that there exists an infinite family of $\theta$-stable bricks of the same dimension $d$. Let us take $\Z \in \Irr(A)$ containing infinitely many such $\theta$-stable bricks and take $X$ to be any such brick. In particular, $X$ does not have an open orbit. Observe that being $\theta$-stable, it follows from \cite[Section 2.3]{AI} that $X$ gives rise to a wall $\Theta_X$ in the wall-and-chamber structure of $K_0(\proj A)$; for details, see Section 2.3 in \cite{AI}. 
We claim that the relative interior of $\Theta_X$ does not intersect the $\tau$-tilting fan of $A$. 
If we assume otherwise, then there exists a $\tau$-rigid module $T$ and a projective module $P$ such that the $g$-vector $\phi$ of the support $\tau$-rigid pair $(T,P)$ lies in the relative interior of $\Theta_X$. We may assume that $P$ has maximal support. This implies that $X$ is $\phi$-stable. Recall that
$$\phi = {\rm dim}_k\Hom(T, -) - {\rm dim}_k\Hom(-, \tau T) - {\rm dim}_k\Hom(P,-).$$
Let $U$ be the Bongartz completion of $T$ (therefore, we have that $U = (T \oplus T',P)$ is support $\tau$-tilting).  Using that $X$ is $\phi$-stable, we get that that $X$ is a simple object in the corresponding wide subcategory $\mathcal{W}_\phi$, which is the same as the category $T^\perp \cap {}^\perp(\tau T) \cap P^\perp$. Let $C = {\rm End}(U)/\langle T \rangle$ be the quotient of the endomorphism algebra of $U$ by the morphisms which factor through $T$. By \cite[Theorem 3.8]{Ja}, there exists an equivalence $F: T^\perp \cap {}^\perp(\tau T) \cap P^\perp \to \modu C$. Now, the brick $X$ is such that $F(X)$ is simple in $\modu C$, hence $F(X)$ is evidently a labeling brick between functorially finite torsion classes in $\modu C$. Therefore, from the above equivalence, it follows that $X$ is also a labeling brick between functorially finite torsion classes in $\modu A$. On the other hand, by \cite[Theorem 6.1]{MP3}, $X$ has an open orbit, which implies the desired contradiction. Hence, the relative interior of $\Theta_X$  does not intersect the $\tau$-tilting fan. Observe that the relative interior of $\Theta_X$ is non-empty because $\Theta_X$ has codimension one and we may assume that the rank of $A$ is at least two. Being a rational polyhedral cone, $\Theta_X$ certainly contains a rational ray $\psi$ in its relative interior. This $\psi$ yields a rational ray outside of the $\tau$-tilting fan, and finishes the proof of the implication $(5) \rightarrow (4)$. 
\end{proof}

Before we derive some interesting consequences of the above results, let us remark that Theorem \ref{Thm: Alg. with generalized standard components} and Corollary \ref{Cor: all bBT conjectures for gen. stan. comp.}, together, imply Theorem C in Section \ref{Section: Introduction}.

\medskip

The above results, together with some known facts on tame algebras and our earlier work, provide new tools for the study of our brick-Brauer-Thrall conjectures. This ultimately allows us to verify all the conjectures from Section \ref{Subsection:bBT Conjs} for a large subfamily of tame algebras. Before we collect these results in the following corollary, let us recall that, for a tame algebra $A$, it is well-known that $A$ admits an infinite family of indecomposables of the same dimension if and only if $\Gamma_A$ contains a homogeneous tube. This is the case if and only if $\Gamma_A$ contains an infinite family of tubes whose mouth are of the same dimension.
The following corollary gives analogous statements for those tame algebras that admit an infinite family of bricks of the same dimension (compare Conjecture \ref{Conj: 2ndbBT}).

\begin{corollary}\label{Cor:equiv. conditions for tame algebras}
If $A$ is a tame algebra of rank $n$, the following are equivalent:
\begin{enumerate}
    \item $A$ admits an infinite family of bricks of the same dimension;
    \item $\Gamma_A$ contains a (generalized) standard stable component;
    \item $\Gamma_A$ contains an infinite family of pairwise Hom-orthogonal tubes whose mouth are of the same dimension;
    \item For some $\theta \in K_0(\proj A)$, there is a rational curve of $\theta$-stable modules;
    \item There is a ray in $\mathbb{Q}^n$ which is outside the $\tau$-tilting fan of $A$.
\end{enumerate} 
\end{corollary}

Before we prove this corollary, let us recall that a stable tube $\mathcal{T}$ in $\Gamma_A$ is known to be standard if and only if $\mathcal{T}$ is generalized standard (see \cite[Lemma 1.3]{Sk1}). Moreover, we remark that some of the implications in the above corollary have already appeared in the literature and are known to the experts. Thus, for the sake of brevity, we only mention some key points and mostly provide references.

\begin{proof}
The implication $(1)\rightarrow (2)$, follows from some fundamental results in \cite{C-B} and \cite{Sk1, Sk2}. In particular, for some $d\in \mathbb{Z}_{>0}$, let $\brick(A,d)$ be an infinite set. Then,
\cite[Theorem D]{C-B} implies that almost all $X \in \brick(A,d)$ are homogeneous, therefore they lie in homogeneous tubes (\cite[Corollary E]{C-B}). On the other hand, a homogeneous stable tube $\mathcal{T}$ contains a brick if and only if its mouth is a brick. By \cite[Lemma 1.3]{Sk1}, this is the case if and only if $\mathcal{T}$ is a (generalized) standard tube.

The implication $(2)\rightarrow (3)$ follows from Proposition \ref{Prop:standard tubes} and Proposition \ref{Prop:2nd bBT gives an infinite family of orthogonal bricks}. More precisely, if $\Gamma_A$ admits a generalized standard tube, by Proposition \ref{Prop:standard tubes}, there exists an infinite family of bricks of the same dimension. Moreover, by Proposition \ref{Prop:2nd bBT gives an infinite family of orthogonal bricks}, we can find an infinite family of pairwise Hom-orthogonal bricks of the same dimension. Hence, by \cite[Corollary E]{C-B}, infinitely many of them lie on the mouth of homogeneous tubes. Such tubes are obviously pairwise Hom-orthogonal, because modules in each tube have a filtration by the quasi-simples lying on the mouth.

To prove $(3)\rightarrow (4)$, first observe that if $\Gamma_A$ contains an infinite family of pairwise Hom-orthogonal tubes whose mouth are of the same dimension, then there exists some $d\in \mathbb{Z}_{>0}$ such that $\brick(A,d)$ is an infinite set (see Corollary \ref{cor:infinite family of orthogonal modules}).
In particular, there exists a brick component $\Z \in \Irr(A)$ with $c(\Z)>0$. Since $A$ is tame, we must have $c(\Z)=1$, which is the case if and only if $\Z$ contains a rational curve consisting of pairwise non-isomorphic bricks (see \cite[Section 2.2]{CC}). 
Again, by \cite[Theorem D]{C-B}, we get that almost all $X \in \Z$ are homogeneous. Thus, by Lemma \ref{Lem: Homogeneous brick is stable}, every such brick $X$ is $\theta_X$-stable. Put $\theta:=\theta_X$, for one of such homogeneous bricks, and let $\Z^{s}_{\theta}$ denote the set of $\theta$-stable modules in $\Z$. Obviously, $\Z^{s}_{\theta}$ is nonempty and dense in $\Z$, hence the desired result follows.

Since the implication $(4)\rightarrow (1)$ follows from the definition, the above arguments show that the first four statements are equivalent.We further note that the equivalence of $(2)$ and $(5)$ follows from Corollary \ref{Cor: all bBT conjectures for gen. stan. comp.}. This completes the proof.
\end{proof}

Let us finish this section with a remark to highlight some connections between the problems and results in this work and some closely related topics, while pointing out some contrasts. Recall that $X$ in $\modu A$ is said to be \emph{$\tau$-rigid} if $\Hom_A(X,\tau X)=0$, and let $\textit{i}\taurigid(A)$ denote the set of all indecomposable $\tau$-rigid modules in $\modu A$.
Thanks to some fundamental results from \cite{DIJ}, for any algebra $A$, it is known that $\textit{i}\taurigid(A)$ and $\brick(A)$ are closely related, and a good knowledge of either of these two sets provides insights to the $\tau$-tilting theory of $A$.
More specifically, the authors gave a ``brick-$\tau$-rigid correspondence", as an explicit injective map from $\textit{i}\taurigid(A)$ into $\brick(A)$, and showed that $A$ is brick-finite if and only if $\textit{i}\taurigid(A)$ is a finite set, which is known to be equivalent to the $\tau$-tilting finiteness of $A$.
Consequently, our results on the study of bricks and their distribution can be seen through the lense of $\tau$-tilting theory and several other closely related areas of research.
Meanwhile, it is known that the maximal number of pairwise Hom-orthogonal components of $\Gamma_A$ containing $\tau$-rigid modules is at most $n$ (\cite[Corollary 2.2]{Sk1}). Observe that the verbatim brick-analogue of this statement is evidently wrong (consider the Kronecker algebra). However, if $A$ is brick-finite, Proposition \ref{Prop:brick-fin and Hom-orthogonality} implies that $\Gamma_A$ contains at most $n$ pairwise orthogonal components. Let us remark that the converse of the latter statement is not true in general. More precisely, if $A$ is a wild hereditary algebra, Baer \cite[Prop. 3.1]{Ba} showed that if $X$ and $Y$ are two regular modules in $\Gamma_A$, then there exists an integer $r$ for which $\Hom_A(X, \tau^r Y) \ne 0$. Hence, no two regular components in $\Gamma_A$ can be orthogonal. Consequently, for a (connected) wild hereditary algebra $A$, although $A$ is evidently brick-infinite, no pair of connected components of $\Gamma_A$ are orthogonal.

\vskip 1cm

\textbf{Acknowledgements.} 
The first-named author was supported by Early-Career Scientist JSPS Kakenhi grant number 24K16908.
The second-named author was supported by the Natural Sciences and Engineering Research Council of Canada and by the Canadian Defence Academy Research Programme. 
The authors would like to thank Lidia Angeleri Hügel, Sota Asai, Osamu Iyama, Rosanna Laking, Francesco Sentieri and Jan Schröer for some helpful discussions about problems related to this project.


\begin{thebibliography}{99999}

	 \bibitem[As]{As} S. Asai, \emph{Semibricks}, International Mathematics Research Notices, Volume 2020, Issue 16 (2020), 4993--5054
  
  \bibitem[AI]{AI} S. Asai, O. Iyama,  
\emph{Semistable torsion classes and canonical decompositions}, arXiv: 2112.14908 (2021).

   
    
	\bibitem[ASS]{ASS} I. Assem, D. Simson, A. Skowro\'nski,  
\emph{Elements of the representation theory of associative algebras},
Volume 1, Cambridge University Press, Cambridge (2006).

    \bibitem[ARS]{ARS} M. Auslander, I. Reiten and S. Smalø, 
\emph{Representation Theory of Artin Algebras, Cambridge Studies in Advanced Mathematics}, 36. Cambridge University Press, Cambridge (1997).

\bibitem[Ba]{Ba} D. Baer, \emph{Wild hereditary Artin algebras and linear methods}, Manuscripta Math 55 (1986), 69--82.

    \bibitem[Bo]{Bo} K. Bongartz, 
\emph{On minimal representation-infinite algebras}, arXiv:1705.10858v4 (2017).

        \bibitem[CC]{CC} C. Chindris, A. Carroll,
\emph{On the invariant theory for acyclic gentle algebras}, Trans. Amer.
Math. Soc. 367 (2015), 3481--3508.

    \bibitem[CKW]{CKW} C. Chindris, R. Kinser, J. Weyman,
\emph{Module Varieties and Representation Type of Finite-Dimensional Algebras}, Int. Math. Res. Not. (2015), 631--650.

\bibitem[C-B]{C-B} W. Crawley-Boevey,
\emph{On tame algebras and bocses}, Proc. London Math. Soc. (1988), 451--483.

    \bibitem[De]{De} L. Demonet, 
\emph{Combinatorics of mutations in representation theory}, Habilitation (2017), Available on Laurent Demonet’s website.

    \bibitem[DIJ]{DIJ} L. Demonet, O. Iyama, G. Jasso,
\emph{$\tau$-tilting finite algebras and $g$-vectors},  Int. Math. Res. Not. (2019), 852--892.

    \bibitem[Do]{Do} M. Domokos, 
\emph{Relative invariants for representations of finite dimensional algebras}, Manuscripta
Math. 108 (2002), 123--133.

    \bibitem[En1]{En1} H. Enomoto,
\emph{Monobrick, a uniform approach to torsion-free classes and wide subcategories},  Adv. Math. 393 (2021).

\bibitem[En2]{En2} H. Enomoto, \emph{Rigid modules and ice-closed subcategories in quiver representations}, J. Algebra 594 (2022), 364--388.

	\bibitem[G+]{G+} F. Gei\ss{}, D. Labardini-Fragoso, J. Schr\"oer,
\emph{Semicontinuous maps on module varieties}, arXiv:2302.02085 (2023).

    \bibitem[HV]{HV} D. Happel, D. Vossieck,
\emph{Minimal algebras of infinite representation type with prepro- jective component}, Manuscripta Mathematica 42, no. 2–3 (1983), 221--243.

   \bibitem[Ja]{Ja} G. Jasso,
\emph{Reduction of $\tau$-tilting modules and torsion pairs}, Int. Math. Res. Not. IMRN 2015,
no. 16 (2015), 7190--237

    \bibitem[Li1]{Li1} S. Liu,
\emph{Infinite radicals in standard Auslander-Reiten components}, J. Algebra 166 (1994),
245--254.

    \bibitem[Li2]{Li2} S. Liu,
\emph{On short cycles in a module category}, Journal of the London Mathematical Society, Volume 51, Issue 1 (1995), 62--74.

   \bibitem[Li3]{Li3} S. Liu,
\emph{Tilted algebras and generalized standard Auslander-Reiten components}, Archiv Math.
(Basel) 61 (1993), 12--19.

    \bibitem[Mo1]{Mo1} K. Mousavand,
\emph{$\tau$-tilting finiteness of minimal representation-infinite algebras}, PhD thesis presented at the University of Montréal (2020).

    \bibitem[Mo2]{Mo2} K. Mousavand,
\emph{$\tau$-tilting finiteness of non-distributive algebras and their module varieties}, Journal of Algebra 608 (2022), 673--690.

    \bibitem[MP1]{MP1} K. Mousavand, C. Paquette,
\emph{Biserial algebras and generic bricks}, arXiv:2209.05696 (2022).

\bibitem[MP2]{MP2} K. Mousavand, C. Paquette,
\emph{Geometric interactions between bricks and $\tau$-rigidity}, arXiv:2311.14863 (2023).

    \bibitem[MP3]{MP3} K. Mousavand, C. Paquette,
\emph{Minimal $\tau$-tilting infinite algebras}, Nagoya Mathematical Journal 249 (2022), 221--238.


\bibitem[MS1]{MS1} P. Malicki, A. Skowroński,
\emph{Algebras with separating almost cyclic coherent Auslander–Reiten components}, J. Algebra 291 (2005), 208--237.

\bibitem[MS2]{MS2} P. Malicki, A. Skowroński,
\emph{On the indecomposable modules in almost cyclic coherent Auslander-Reiten components}, J. Math. Soc. Japan 63 (2011), 1121--1154.

    \bibitem[MS3]{MS3} P. Malicki, A. Skowroński,
\emph{The structure and homological properties of generalized standard Auslander–Reiten components}, Journal of Algebra 518, (2019), 1--39.

\bibitem[MS]{MS} F. Marks, J. Šťovíček, \emph{Torsion classes, wide subcategories and localisations}, Bulletin of the LMS 49 (2017), no. 3, 405--416.

    \bibitem[Pf1]{Pf1} C. Pfeifer, 
\emph{A generic classification of locally free representations of affine GLS algebras}, 	arXiv:2308.09587 (2023).

    \bibitem[Pf2]{Pf2} C. Pfeifer, 
\emph{Remarks on $\tau$-tilted versions of the second Brauer-Thrall Conjecture}, arXiv:2308.09576 (2023).

    \bibitem[Pl]{Pl} P-G. Plamondon, 
\emph{Generic bases for cluster algebras from the cluster category}, Int. Math. Res. Not. (2013), 2368–2420.

    \bibitem[PY]{PY} P-G. Plamondon, T. Yurikusa, with an appendix by B. Keller,
\emph{Tame algebras have dense g-fans}, Int. Math. Res. Not. (2023), 2701--2747.

    \bibitem[Re]{Re} C. Riedtmann, 
\emph{Many algebras with the same Auslander-Reiten quiver}, Bull. London Math. Soc. 15 (1983), 43--47.

    \bibitem[STV]{STV} S. Schroll, H. Treffinger, Y. Valdivieso, \emph{On band modules and $\tau$-tilting finiteness}, Mathematische
Zeitschrift 299 (2021), 4567--4575.

    \bibitem[SS]{SS} D. Simson, A. Skowroński,
\emph{Elements of the Representation Theory of Associative Algebras. Vol. 2}, London Mathematical Society Student Texts 71. Cambridge: Cambridge University Press (2007). Tubes and concealed algebras of Euclidean type.

    \bibitem[Sk1]{Sk1} A. Skowroński,
\emph{Generalized canonical algebras and standard stable tubes}, Colloq. Math., 90 (2001), 77--93.

    \bibitem[Sk2]{Sk2} A. Skowroński,
\emph{Generalized standard Auslander–Reiten components}, J. Math. Soc. Japan 46 (1994), 517--543

\bibitem[Zh]{Zh} Y. Zhang
\emph{The structure of stable components}, Can. J. Math. 43 (1991), 652--672.

\end{thebibliography}
\end{document}